\documentclass[10pt]{article}

\usepackage{amsfonts}
\usepackage{amsmath}
\usepackage{amssymb}
\usepackage{amsthm}
\usepackage[alphabetic]{amsrefs}
\usepackage{calc}
\usepackage{graphicx}
\usepackage[hidelinks]{hyperref}
\usepackage[left=1in,top=1in,right=1in,foot=1in]{geometry}
\usepackage{todonotes}
\usepackage{microtype}
\usepackage{tikz-cd}

\newcommand{\pair}[2]{\left\langle #1 , #2\right\rangle}

\def\into{\mathrel{\hookrightarrow}}
\def\onto{\mathrel{\twoheadrightarrow}}

\newcommand{\bangles}[1]{\left\langle #1 \right\rangle}
\newcommand{\sett}[1]{\left\{#1\right\}}

\newcommand{\rquotient}{\backslash}

\DeclareMathOperator{\End}{End}

\newcommand{\HomOver}[3]{\mathrm{Hom}_{#1}(#2, #3)}

\DeclareMathOperator{\Mat}{Mat}

\newcommand{\id}{\mathrm{id}}

\newcommand{\Bb}{\mathcal{B}}

\newcommand{\pt}{\mathrm{pt}}

\newcommand{\X}{\mathcal{X}}
\newcommand{\HH}{\mathbf{H}}

\newcommand{\Ff}{\mathcal{F}}
\newcommand{\Gg}{\mathcal{G}}

\newcommand{\dInt}[2]{\int {#1} \,\mathrm{d} {#2}}
\newcommand{\dIntOver}[3]{\int_{#1} {#2} \,\mathrm{d} {#3}}

\DeclareMathOperator{\vol}{vol}

\newcommand{\trace}[2]{\mathrm{trace}\left({#1}\,,{#2}\right)}

\newcommand{\N}{\mathbb{N}}
\newcommand{\Z}{\mathbb{Z}}
\newcommand{\C}{\mathbb{C}}

\newcommand{\Oo}{\mathcal{O}}

\newcommand{\g}{\mathfrak{g}}

\newcommand{\bB}{\mathfrak{b}}

\newcommand{\Nn}{\tilde{\mathcal{N}}}

\newcommand{\Waff}{W_{\mathrm{aff}}}

\newcommand{\cc}{\mathbf{c}}

\newcommand{\GL}{\mathrm{GL}}

\newcommand{\SL}{\mathrm{SL}}

\newcommand{\Gm}{{\mathbb{G}_\mathrm{m}}}

\newcommand{\Sn}{\mathfrak{S}}

\newcommand{\St}{\mathrm{St}}

\newcommand{\bq}{\mathbf{q}}

\newtheorem{theorem}{Theorem}
\newtheorem*{theorem*}{Theorem}
\newtheorem{cor}{Corollary}
\newtheorem{prop}{Proposition}
\newtheorem{lem}{Lemma}

\theoremstyle{definition}
\newtheorem{dfn}{Definition}
\newtheorem*{dfn*}{Definition}
\theoremstyle{remark}
\newtheorem{ex}{Example}
\newtheorem*{ex*}{Example}
\newtheorem{rem}{Remark}

\providecommand{\keywords}[1]{\textbf{\textit{Keywords---}} #1}

\title{A positivity property in the based ring of the lowest two-sided cell}
\date{\today}
\author{Stefan Dawydiak \thanks{School of Mathematics and Statistics, University of Glasgow, University Place,
Glasgow G12 8QQ; email \texttt{stefan.dawydiak@glasgow.ac.uk}}}

\begin{document}
\maketitle

\begin{abstract}
Let $\Waff$ be an extended affine Weyl group and $\HH$ and $J$ be the corresponding affine and asymptotic Hecke algebras with standard bases $\sett{T_x}$ and $\sett{t_w}$, respectively. Viewing $J$ as a subalgebra of the $\bq^{-\frac{1}{2}}$-adic completion of $\HH$, we give formulas for the coefficient of $T_x$ in $t_w$ for various $x$ and $w$ in the lowest two-sided cell, in terms of generalized exponents of the Langlands dual group, under a hypothesis on the left cell containing $w$. In particular our results hold for the canonical left cell. For such $w$ we also define a seemingly new positive basis for the corresponding subring of $J$. For $\GL_n$, we give partial results for some other cells.
\end{abstract}

\section{Introduction}
Let $G$ be a connected reductive group defined and split over a non-archimedean local field $F$, with root datum $(X^*, \Phi, X_*, \Phi^\vee)$. Choose a Borel subgroup of $G$ and let $I$ be the corresponding Iwahori subgroup. It is well known that for the affine Hecke algebra $\HH$ of $\Waff$ over 
$\C[\bq^{\pm\frac{1}{2}}]$, $\HH|_{\bq=q}\simeq C_c(I\rquotient G(F)/I)$, where $q$ is the cardinality of the residue field of $F$. Therefore $\HH$ relates to the representation theory of $G(F)$.

In \cite{affineII}, Lusztig constructed the asymptotic Hecke algebra $J$ in particular for affine Weyl groups, and equipped it with a basis $\sett{t_w}_{w\in \Waff}$ and an injection $\phi\colon H\into J[\bq,\bq^{-1}]$, such that %
\[
(\phi\circ{}^\dagger(-))^{-1}(t_w)=\sum_{x\in \Waff}a_{x,w}T_x
\]
for $a_{x,w}\in\Z((\bq^{-1}))$, where ${}^\dagger(-)$ is the involution of $H$ recalled below. In \cite{Plancherel}, we used the interpretation of $J$ in terms of harmonic analysis on $G(F)$ of \cite{BK} to show that $a_{x,w}$ is a rational function with denominator dividing a bounded power of the Poincar\'{e} polynomial $P_W(\bq)$ of $W$. In the present paper, we give formulas for the numerators of $a_{x,w}$ in special cases, giving the representation-theoretic information hoped for in the paragraph following \cite[Conj. 1.2]{SL2}.
Let $\widetilde{G^\vee}$ be the universal cover of $G^\vee$.
\begin{theorem}
\label{thm K theory formulas}
Let $d, d'$ be distinguished involutions in the lowest two-sided cell $\cc_0$ corresponding to $u,u'\in W$ via Shi's parametrization \cite{Shi}, let $t_w\in t_d Jt_{d'}$ correspond to a dominant weight $\lambda$ under the same. 
Let $\sett{\Oo(x_u)}_{u\in W}$ be Steinberg's $K(\pt/\widetilde{G^\vee})$-basis of $K(G^\vee/B^\vee/\widetilde{G^\vee})$.
If the dual class to $\Oo(x_{u'})$ with respect to \eqref{eqn tor pairing} is represented by a shifted line bundle $\Oo(y_{u'})[n(u')]$, then for all $\gamma$ sufficiently dominant, we have
\begin{align}
t_{w}(I\varpi^{\gamma}I)
&=
(-1)^{\ell(\omega(\gamma)_f)+\ell(w_0)+n(d')}
\frac{\bq^{-\frac{\ell(\gamma)}{2}}}{P_W(\bq)}
\sum_{i}\dim\HomOver{G^\vee}{V(\lambda)}{V(\gamma-x_{u}-y_{u'}-2\rho)\otimes\Oo(\mathcal{N}^\vee)_i}\bq^{-i}
\label{tdJtd' dominant formula denom}
\\
&=
\frac{(-1)^{\ell(\omega(\gamma)_f)+\ell(w_0)+n(d')}\bq^{-\ell(w_0)}}{\prod_{i=1}^{r}(1-\bq^{-1})}
\bq^{-\frac{\ell(\gamma)}{2}}
\sum_{i}\dim\HomOver{G^\vee}{V(\lambda)}{V(\gamma-x_{u}-y_{u'}-2\rho)\otimes\Oo(\g^\vee)_i}\bq^{-i},
\label{tdJtd' dominant formula no denom}
\end{align}
%
%
where $\Oo(\mathcal{N}^\vee)_i$ and $\Oo(\g^\vee)_i$ are the space of homogeneous degree $i$ global functions on $\mathcal{N}^\vee$ and $\g^\vee$, respectively, $r$ is the semisimple rank of $\g^\vee$ and $\omega(\gamma)_f$ is the $W$-component of the projection $\omega(\gamma)$ of $\gamma$ onto $\pi_1(G)$.
\end{theorem}
%
%
\begin{rem}
Of course, $\Oo(\mathcal{N}^\vee)$ depends only on $G^\vee/Z(G^\vee)$, and the multiplicity in \eqref{tdJtd' dominant formula denom} is zero unless $V(\lambda)$ and $V(\gamma-x_u-y_{u'}-2\rho)$ have the same central character. Together with the sign $(-1)^{\ell(\omega(\gamma)_f)}$, this matches the formulas of \cite{Plancherel} and explains Remark 3.8 of \textit{op. cit.}. 
\end{rem}

\begin{rem}
When $\gamma$ is too close to a wall, it will be clear from the proof that the signs in Theorem \ref{thm K theory formulas} can change. For instance when $\gamma=0$, both signs disappear. (C.f. \cite[(3.3)]{Plancherel}.)
\end{rem}

In the completion $K(\pt/ G^\vee)((\bq))$ of $K(\pt/G^\vee\times\Gm)$, we have $R\Gamma(\Oo(\mathcal{N}^\vee))=R\Gamma(\mathcal{K}^{-1})$, where $\mathcal{K}=\prod_{\alpha\geq 0}(1-\bq\Oo(\alpha))$ is the class of the Koszul complex for $\Bb^\vee\into\Nn^\vee$. Hence the values \eqref{tdJtd' dominant formula denom} are closely related to the $q$-analogue of the Kostant partition function. This is unsurprising as our computation boils down to the action of the spherical Hecke algebra on the $K$-invariants of unramified principal series and the connection of the latter to the functions $c_\lambda\in \mathcal{S}^K$ of \cite[\S 3.12]{BKBasicAffine}. As the values in \eqref{tdJtd' dominant formula denom} are essentially Macdonald inner products of Schur functions, we also obtain
\begin{cor}
\label{cor symm and pos dc0}
Let $\mathcal{L}(d)$ be the left descent set of $d$.
Under the hypotheses of Theorem \ref{thm K theory formulas}, we have
\begin{enumerate}
\item[(a)] 
\[
\bq^{\frac{-\ell(\gamma)+\ell(\lambda+x_{u}+y_{u'}+2
\rho)}{2}}t_\lambda(\varpi^\gamma)=t_{\gamma-x_u-y_{u'}-2\rho}(\varpi^{(\lambda+x_u+y_{u'}+2\rho)});
\]
\item[(b)]
Moreover, writing 
\[
t_w=\sum_{x}a_{x,w}T_x,~~~t_w\in t_{w_0}Jt_{w_0},
\]
the $a_{x,w}\in \Z[\bq^{-1}]$ have positive coefficients for $x=x_f\gamma$ with $x_f\in\bangles{\mathcal{L}(d)}$ and $\gamma$ a sufficiently dominant translation element, up to a controlled sign. That is,
\[
(-1)^{\ell(\omega(\gamma)_f)+\ell(w_0)+n(d')}
a_{x,w}\in\frac{1}{P_{W_f}(\bq)}\N[\bq^{-1}].
\]
\end{enumerate}
\end{cor}
The symmetry in (a) resembles slightly the symmetry of Macdonald polynomials.

In general, the class dual to some $\Oo(x_{u'})$ does not seem to be represented by a line bundle. However, this is always the case for at least one distinguished involution.
\begin{cor}
\label{cor tw0 Jtdc0}
The hypotheses of Theorem \ref{thm K theory formulas} hold for $d'=d_{\cc_0}$ the canonical distinguished involution in $\cc_0$. 
\begin{enumerate}
\item 
For $t_w\in t_{d_{\cc_0}}Jt_{d_{\cc_0}}$, Corollary \ref{cor symm and pos dc0} becomes
\[
t_{w}(I\varpi^{\gamma}I)=
(-1)^{\ell(\omega(\gamma)_f)}\frac{\bq^{-\frac{\ell(\gamma)}{2}}}{P_W(\bq)}
\sum_{i}\dim\HomOver{G^\vee}{V(\lambda)}{V(\gamma)\otimes\Oo(\mathcal{N}^\vee)_i}\bq^{-i}
\]
and
\[
\bq^{\frac{-\ell(\gamma)+\ell(\lambda)}{2}}t_\lambda(\varpi^\gamma)=t_{\gamma}(\varpi^{\lambda});
\]
\item
For $t_w\in t_{w_0} Jt_{d_{\cc_0}}$, we obtain 
\begin{enumerate}
\item 
$t_w$ is $K\times I$-invariant, so is determined by its values on translation elements;
\item
For $\gamma$ dominant and not on a wall, if $w$ corresponds to $\lambda$, we have
\[
t_{w}(\varpi^{\gamma})
=
(-1)^{\ell(\omega(\gamma)_f)}\frac{\bq^{-\frac{\ell(\gamma)}{2}}}{P_W(\bq)}
\sum_{i}\dim\HomOver{G^\vee}{V(\lambda)}{V(\gamma-\rho)\otimes\Oo(\mathcal{N}^\vee)_i}\bq^{-i};
\]
\item
For $\gamma$ dominant and $x\in W$ a product of commuting simple reflections, we have 
\[
t_\lambda\left(\varpi^{x(\gamma)}\right)=(-\bq)^{\ell(x)}t_\lambda(\varpi^x).
\]
\end{enumerate}
\end{enumerate}
\end{cor}

\subsection{A second positive basis of $t_d Jt_{d'}$}
The affine Hecke algebra is equipped with two canonical bases in addition to its standard basis $\sett{T_w}_{w\in \Waff}$. On one hand, there is the Kazhdan-Lusztig basis $\sett{C'_w}$. This basis encodes geometric information about the affine flag variety, in both the support and the values of $C'_w$ as a function on $I\rquotient G(F)/I$. On the other hand, the elements of Bernstein's basis $\sett{T_w\theta_\lambda}_{w\in W, \lambda\in X_*}$ of $H$ are adapted to spectral questions: it is convenient to build modules and compute traces via the Bernstein presentation. Moreover, the constructions are naturally in terms of the Langlands dual group $G^\vee$.
Lusztig conjectured in \cite[\S 10]{affineIV} a
version of the Bernstein presentation for $J$. His conjecture was proven in special cases in \cite{XiLowIII, XiTypeA}, and in general in \cite{BO}. However, this presentation did not come with a second basis of $J$. Rather, the $t_w$ basis is already nicely adapted to this presentation, \textit{e.g.} via \eqref{eqn based ring iso}, despite its construction using essentially the $C'_w$-basis. 

The proof of Theorem \ref{thm K theory formulas} uses this asymptotic Bernstein presentation essentially, and a byproduct gives the existence of a second basis any subring $t_dJt_{d'}$ to which the theorem applies. 
\begin{cor}
\label{cor KL}
Suppose that $t_d Jt_{d'}$ for $d,d'\in\cc_0$ satisfies the hypotheses of Theorem \ref{thm K theory formulas}. Then there is a basis $f_\mu$ of  $t_d Jt_{d'}$ consisting of functions given by 
\[
f_\mu(I\varpi^{\gamma} I)
=
\frac{1}{P_W(\bq)}
\bq^{-\ell(\gamma)}
(-1)^{\ell(\omega(\gamma)_f)}
P_{\mu,\gamma-x_d-y_{d'}-2\rho}(\bq),
\]
for $\gamma$ sufficiently dominant, where the $P_{\mu, \lambda}$ are spherical Kazhdan-Lusztig polynomials. For $d=w_0$, $d'=d_{\cc_0}$, these functions satisfy the symmetry property of Corollary \ref{cor tw0 Jtdc0}, 2(c).
\end{cor}
The elements of the basis $\sett{f_\mu}_{\mu}$ of $t_dJ_0t_{d'}$ therefore have some restriction on their supports. Note that for $G=\SL_2$, the $t_w$ are all supported everywhere \cite{SL2}, whereas at least $f_\mu(\varpi^{\gamma})=0$ if $\mu\not\leq\gamma-x_d-x_{d'}-2\rho$. 
The $f_\mu$ also have some connection to the geometry of the affine Grassmannian, an object appearing in various categorifications in terms of constructible sheaves.

Existing categorifications of $J$, namely \cite{Propp}, \cite{BKK}, and partially \cite{coh}, are all in terms of coherent sheaves. One also hopes for a categorification of $J$ in terms of constructible sheaves, and an equivalence in the style of Bezrukavnikov's equivalence \cite{tworel}. At least naively, one therefore expects a second basis of $J$ in addition to the spectrally-adapted $t_w$ basis. The bases of the subrings to which Corollary \ref{cor KL} applies may be subsets of the former.

In Section \ref{section rec rel}, we explain that the functions $f_\mu$ all satisfy finite linear recurrences with constant coefficients. This may help with an asymptotic characterization of $J_0$, and is an independent reason for our interest in the $f_\mu$ basis.
\begin{ex}
In type $\tilde{A}_{n}$, we have $P_{\mu, N\varpi_i}(q)=1$ for $i=1,n$ and $N\gg 0$, by \cite{spiral}. Hence in this case the functions $f_\mu$ behave very simply in the corresponding translates of these directions, and in the further directions implied by Corollaries \ref{cor symm and pos dc0} and \ref{cor tw0 Jtdc0}. For instance, we get by Corollary \ref{cor symm and pos dc0} (a) that the functions $t_{\lambda}$ for $\lambda$ an appropriate translate of $N\varpi_i$, $i=1,n$ are given essentially by inverse Kazhdan-Lusztig polynomials. (See the proof of Proposition 3.13 of \cite{Plancherel}.)
\end{ex}

\subsection{Other cells}
The methods of this paper apply in very slight generality beyond the lowest cell. Namely, for $G=\GL_n$, a trick using a coincidence of Plancherel measures gives similar formulas to Theorem \ref{thm K theory formulas} for $J_\cc$, where $\cc$ corresponds to a homogeneous Levi subgroup, but only for values at $1$. The trick fails whenever $M_\cc$ has blocks of different sizes. However, the cancellation implied by the formula
\[
1_J=1_H=T_1=\sum_{d\in\mathcal{D}}t_d
\]
suggests that the other cells might still satisfy similar formulas.

To understand other values would require more knowledge about $\phi_\cc$. For example, \cite{BKK} give coherent realizations of $J_\cc$ for general $\cc$ generalizing \cite{XiLowIII}, but do not describe $\phi_\cc$ in geometric terms. 

In \cite{BKParab}, Braverman-Kazhdan defined parabolic Schwartz spaces $\mathcal{S}(X_P)$ analogous to $\mathcal{S}$, with $\mathcal{S}(X_P)^I$ again having a $K$-theoretic model generalizing $\Bb^\vee=\Nn^\vee\times_{\g^\vee}\sett{0}$. Better understanding of these spaces could be another way of repeating the strategy of the present paper. The difficulty in this perhaps reflects that the discrete and supercuspidal supports of the 
 tempered representations acted upon by other summands $J_\cc$ differ for $\cc\neq \cc_0$.
 
Finally, we check that our formulas are consistent with \cite{SL2}.
\begin{ex}
Let $G=\SL_2$. It is easy to see that the Steinberg basis consists of self-dual elements up to sign, and hence that 
\[
t_{w_0}Jt_{w_0}\into K(\Bb^\vee\times\Bb^\vee/G^\vee\times\Gm)
\]
via 
\[
t_{w_0\lambda}\mapsto V(\lambda)\Oo\boxtimes\Oo
\]
and for the affine simple reflection $s_a=d_{\cc_0}$, 
\[
t_{s_a}J_0t_{s_a}\into K(\Bb^\vee\times\Bb^\vee/G^\vee\times\Gm)
\]
via
\[
t_\lambda\mapsto -V(\lambda)\Oo(-1)\boxtimes\Oo(-1).
\]

For Theorem \ref{thm K theory formulas} (c), recall that $\mathcal{K}=\Oo-\bq\Oo(2)\in K(\Bb/G^\vee\times\Gm)$ is the class of the Koszul complex for $\Bb^\vee\into\Nn^\vee$. Then we have
\[
\Oo(\mathcal{N}^\vee)=\frac{1}{R\Gamma(\mathcal{K})}=1+\bq V(2)+\bq^{2}V(4)+\bq^{3}V(6)+\cdots
\]
so that $\Oo(\mathcal{N}^\vee)_i=V(2i)$. In particular, \eqref{tdJtd' dominant formula denom} is a power of $\bq$, such that 
\begin{equation}
\label{eqn SL2 w0 decay property}
t_{\lambda w_0}(\varpi^{\gamma+\alpha^\vee})=
\bq^{-\ell(\alpha^\vee)}t_{\lambda w_0}(\varpi^{\gamma})
=
\bq^{-2}t_{\lambda w_0}(\varpi^{\gamma})
\end{equation}
for $\gamma$ dominant, in keeping with \cite[Cor. 2.12]{SL2}.

The image of $T_s$ in $K(\Bb\times\Bb/G^\vee\times\Gm)$ is 
\begin{equation}
\label{eqn Ts SL2 image}
(\mathcal{K}\boxtimes\Oo_{\Bb^\vee})\otimes(-\Oo_{\Bb^\vee}(-1)\boxtimes\Oo_{\Bb^\vee}(-1))-\Delta_*\Oo_{\Bb^\vee}
=
-\Oo(-1,-1)+\bq\Oo(2,-1)-\Delta_*\Oo
\end{equation}
by \cite[Lem. 8.11]{BasesI}. Hence we have $\Phi(T_s)\star\Oo_{\Bb^\vee}=-\Oo_{\Bb^\vee}$ and $\Phi({}^\dagger T_s)\star\Oo_{\Bb^\vee}=\bq\Oo_{\Bb^\vee}$, and we recover that $t_{w_0}$ is left $K$-invariant.

In this case we can go beyond the symmetry property in Corollary \ref{cor tw0 Jtdc0}, 2 (c), and recover a symmetry property for $t_{w_0}$. Indeed, by \eqref{eqn Ts SL2 image}, we have for $\gamma$ dominant
\[
\Phi(T_s)\star\Oo(-\gamma)=-R\Gamma\left(\Oo(-\gamma-1)\right)\otimes\Oo(-1)+\bq R\Gamma\left(\Oo(-\gamma-1)\right)\otimes\Oo(1)-\Oo(-\gamma).
\]
Applying $R\Gamma$, we obtain
\begin{align*}
\bq R\Gamma\left(\Oo(-\gamma-1)\right)\otimes V(1)-R\Gamma\left(\Oo(-\gamma)\right)
&=
-\bq V(\gamma+1-2)\otimes V(1)+V(\gamma-2)
\\
&=
-\bq V(\gamma-2)-qV(\gamma)+V(\gamma-2),
\end{align*}
which implies that for $\bq=q>1$, we have, by the Plancherel formula,
\begin{align*}
t_{w_0\lambda}(\varpi^{-\gamma})&=q^{-\ell(\gamma)}\dIntOver{\mathrm{P.S.}}{\mathrm{trace}\left(V(\lambda)\otimes\Oo\boxtimes\Oo\star\Phi({}^\dagger T_\gamma)\right)}{\pi}
\\
&=
q^{-\frac{\ell(\gamma)}{2}}\dIntOver{\mathrm{P.S.}}{\mathrm{trace}\left(V(\lambda)\otimes\Oo\boxtimes\Oo\star\Phi({}^\dagger \theta_\gamma)\star\Oo\boxtimes\Oo\right)}{\pi}
\\
&=
q^{-\frac{\ell(\gamma)}{2}}\dIntOver{\mathrm{P.S.}}{\mathrm{trace}\left(V(\lambda)\otimes\Oo\boxtimes\Oo\star\Phi(T_s\theta_{-\gamma)}T_s^{-1})\star\Oo\boxtimes\Oo\right)}{\pi}
\\
&=
-q^{-\frac{\ell(\gamma)}{2}}\dIntOver{\mathrm{P.S.}}{\mathrm{trace}\left(V(\lambda)\otimes\Oo\boxtimes\Oo\star\Phi(T_s\theta_{-\gamma})\right)}{\pi}
\\
&=
q^{-\frac{\ell(\gamma)}{2}}q\dIntOver{\mathrm{P.S.}}{V(\lambda)V(\gamma-2)}{\pi}
+q^{-\frac{\ell(\gamma)}{2}}q\dIntOver{\mathrm{P.S.}}{V(\lambda)V(\gamma)}{\pi}
-q^{-\frac{\ell(\gamma)}{2}}\dIntOver{\mathrm{P.S.}}{V(\lambda)V(\gamma-2)}{\pi}
\\
&=
-q t_{w_0\lambda}(\varpi^\gamma)
-q^2 t_{w_0\lambda}(\varphi^{\gamma+2})
+t_{w_0\lambda}(\varpi^{\gamma})
\\
&=
-q t_{w_0\lambda}(\varpi^{\gamma}),
\end{align*}
where the integral is over the tempered principal series of $H$, and the cancellation on the last line follows from \eqref{eqn SL2 w0 decay property}.
This symmetry property matches the formulas in \cite[Cor. 2.12]{SL2} and explains them. However, necessity of the cancellation shows that $t_{w_0} Jt_{w_0}$ is more complicated in general than $t_{w_0}Jt_{d_{\cc_0}}$.

To compare with \cite{SL2} for off-diagonal elements in $t_d Jt_{d'}$,  the coincidence of signs in Corollary \ref{cor symm and pos dc0} statements 1 and 2 (b) indicates that one must twist by the automorphism of $K(\Bb\times\Bb/G^\vee)$ induced by swapping the factors. (For example, \cite{Nie} differs from \cite{XiLowIII} by this swap.)
\end{ex} 
\subsection{Acknowledgements} 
The author thanks Roman Bezrukavnikov for helpful discussions and for posing the question that prompted this work, and Per Alexandersson, Alexander Braverman, Arnaud Eteve, Anatole Kirillov, Dmitry Kubrak, Alexandre Minets, Dinakar Muthiah, Arun Ram, David Schwein, Mark Shimozono, and Catharina Stroppel for helpful conversations. This research was partially supported by NSERC, and by the Engineering and Physical Sciences Research Council grant UKRI167 ``Geometry of double loop groups." 

\section{Notation and conventions}
\subsection{The affine and asymptotic Hecke algebras}
\label{subsection K-theoretic realizations}
Let $G$ be as above, and let $\Waff=W_f\ltimes X^*(T)$ be its extended affine Weyl group.
Let $\HH$ be the corresponding affine Hecke 
algebra over $\mathcal{A}=\Z[\bq^{1/2},\bq^{-1/2}]$ with standard basis $\{T_w\}_{w\in\Waff}$ and relations  
$T_wT_{w'}=T_{ww'}$ if $\ell(ww')=\ell(w)+\ell(w')$ and 
$(T_s+1)(T_s-\bq)=0$ for all $s\in S$, where $S\subset\Waff$ is the set of simple reflections and $\ell$ is
the length function on $\Waff$. Let $\sett{C_w}_{w\in\Waff}$ and 
$\sett{C'_w}_{w\in \Waff}$ be the Kazhdan-Lusztig bases of $\HH$ \cite{KL79}, and consider the involution of $\HH$ given by ${}^\dagger T_w=(-1)^{\ell(w_f)}\bq^{\ell(w)}T_{w^{-1}}^{-1}$; it obeys ${}^\dagger C_w=(-1)^{\ell(\omega(x)_f)+\ell(x)}C'_w$ \cite[Lem 1.10]{Plancherel}. 
Here, for $x\in \Waff$, $\omega(x)\in \pi_1(G)$ labels the $W_f\ltimes\Z\Phi^\vee$-coset of $\Waff$ containing $x$ and we write $\omega(x)=\omega(x)_f\omega(x)_t\in W\ltimes X_*$. 

Let $h_{x,y,z}$ be the 
structure constants for the $C_w$-basis. Either Kazhdan-Lusztig basis induces the same notion of cells in $W$: we
say that $x\leq_L y$ if 
$C'_x$ appears in $hC'_y$ for some $h$, and similarly for the relation $x\leq_Ry$. The equivalence classes
induced by the transitive closures of these relations are the left and right cells of $W$. The two sided
cells are given by the coarser relation where $x\leq_{LR}y$ if $x\leq_L y$ or $x\leq_R y$. Recall that each one-sided cell contains a unique distinguished involution \cite[Thm. 1.10]{affineII}, and write $\mathcal{D}$ for the (finite) set of distinguished involutions. Recall also that by \cite{LX}, every two-sided cell $\cc$ contains a unique distinguished involution $d_\cc$ which is shortest in its double coset $W_fd_\cc W_f$.
\begin{dfn}
\label{def canonical}
We say that $d_{\cc_0}$ is the \emph{canonical} distinguished involution in $\cc_0$, following the terminology of \cite{LX}.
The left cell of $\cc_0$ containing $d_{\cc_0}$ is the \emph{canonical left cell} of $\cc_0$.  By the proof of Prop. 4.6 of \cite{XiLowI}, we have $d_{\cc_0}=w_0\ltimes(-2\rho)$.
\end{dfn}
\begin{ex}
If $G=\SL_2$, then $d_{\cc_0}$ is the affine simple reflection in $\Waff$. The canonical left cell in $\cc_0$ consists of all $w$ with reduced expression of the form $w=rd_{\cc_0}$ for some $r\in\Waff$.
\end{ex}
Let $J$ be 
Lusztig's asymptotic Hecke algebra with basis $\{t_z\}$ as in \cite{affineII}, and let $J_0$ be 
its direct summand corresponding to the lowest cell $\cc_0$. Let 
\[
\phi_0(C_w)=\sum_{\substack{d\in \cc_0\\ z\sim_L d}}h_{w,d,z}t_z
\]
be the composition $H\into J[\bq^{\frac{\pm 1}{2}}]\onto J_0[\bq^{\frac{\pm 1}{2}}]$ of Lusztig's homomorphism \cite{affineII} and the projection onto $J_0$. 

Let $\Bb^\vee$ and $\St$ be the flag variety and Steinberg variety of $G^\vee$, respectively.
We identify $\HH\simeq K(\St/G^\vee\times\Gm)$ via the isomorphism of \cite[\S 7]{BasesI} where $\mathcal{A}=K(\Gm\rquotient\pt)$ with $\bq^{1/2}=\id_{\Gm}$, and 
$\Gm$ scales the fibres of $\Nn^\vee$ by $\bq^{-1}$. Then, by \cite[Thm. 2.5]{XiLowIII} and \cite{Nie}, the diagram
\begin{equation}
\label{eqn Xi diagram}
\begin{tikzcd}
J_0[\bq^{\pm\frac{1}{2}}]\arrow[d, hook]&H\simeq K(\St/G^\vee\times\Gm)\arrow[l, hook, "\phi_0" above]\arrow[r, hook, "\Phi"]&K(\Bb\times\Bb/G^\vee\times\Gm)\arrow[d, hook]
\\
\Mat_{\#W_f}(K(\widetilde{G^\vee}\times\Gm\rquotient\pt))\arrow[rr, "\mathrm{Ad}(A)"]&&\Mat_{\#W_f}(K(\widetilde{G^\vee}\times\Gm\rquotient\pt))
\end{tikzcd}
\end{equation}
of algebra morphisms commutes, $\Phi$ is the morphism induced by pushforward followed by restriction along the maps 
\[
\St\into \Nn\times\Bb^\vee\hookleftarrow \Bb^\vee\times\Bb^\vee,
\]
$\widetilde{G^\vee}$ is the universal cover of $G^\vee$ and the vertical maps and the matrix $A$ are as in \cite{XiLowIII}. This diagram therefore identifies 
\[
J_0\otimes\mathcal{A}\overset{\sim}{\to} K(\Bb^\vee\times\Bb^\vee/G^\vee\times\Gm).
\]
The identification works as follows. Recall that, by the Pittie-Steinberg theorem \cite{Steinberg}, $K(\Bb^\vee/\widetilde{G^\vee})$ is a free $K(\pt/\widetilde{G^\vee})$-module of rank $\# W_f$ with an explicit basis of line bundles $\sett{\Oo(x_u)}_{u\in W_f}$, where
\begin{equation}
\label{eqn steinberg weights dfn}
x_u=u\left(\sum_{\substack{\varpi_i \\ u\varpi_i<0}}\varpi_i\right)\in X_*,
\end{equation}
where the $\varpi_i$ range over the fundamental dominant weights.

Moreover, the pairing 
\begin{equation}
\label{eqn tor pairing}
\pair{-}{-}\colon K(\Bb^\vee/ G^\vee)\otimes K(\Bb/G^\vee)\to K(\pt/G^\vee)
\end{equation}
given by $\pair{\Ff}{\Gg}=R\Gamma\left(\Ff\otimes\Gg\right)$ is perfect \cite{KLDeligneLanglands}; let $\sett{[\Oo(x_u)]^*}_{u\in W_f}$ be the dual basis. Then if $t_w\in t_d J_0 t_{d'}$ with $d,d'$ corresponding to $u,u'\in W_f$ and $w$ to a dominant weight $\lambda$ of $\tilde{G^\vee}$ via Shi's parametrization \cite{Shi} of $\cc_0$ as recalled in \cite[\S 1.2]{Nie}, we have 
\begin{equation}
\label{eqn based ring iso}
t_w\mapsto V(\lambda)\otimes\Oo(x_u)\boxtimes[\Oo(x_u)]^*\in J_{G^\vee}(\Bb\times\Bb)\simeq J_0,
\end{equation}
such that the right hand side is $G^\vee$-equivariant, not just $\tilde{G}^\vee$-equivariant. Note that $\Oo(x_u)\boxtimes[\Oo(x_u)]^*$ is equivariant for any member the isogeny class of $G^\vee$, but that $\Oo(x_u)\boxtimes[\Oo(x_{u'})]^*$ for $u\neq u'$ might not be \cite[\S 8.2]{XiTypeA}.
(By \cite{rigid}, some equivariance will fail if and only if there is a reducible unramified tempered principal series representation of $G(F)$.)

In general, the basis $\sett{[\Oo(x_u)]^*}_{u\in W_f}$ does not consist of line bundles. However, we do have
\begin{lem}
\label{Lem dual classes}
\begin{enumerate}
\item[(a)]
The class $[\Oo(-\rho)]$ always belongs to the Steinberg basis. Its dual class is $\left[\Oo(-\rho)\left[\ell(w_0)\right]\right]$.
\item[(b)]
Under \eqref{eqn based ring iso}, $t_{d_{\cc_0}}\mapsto\Oo(-\rho)\boxtimes\Oo(-\rho)[\ell(w_0)]$.
\end{enumerate}
\end{lem}
\begin{proof}
For (a), by construction, all $x_u$ except $-\rho$ lie on walls. This implies the second statement. For (b), it suffices by (a) that $d_{\cc_0}$ corresponds to $u=w_0$ in the parametrization of \cite{Shi}. This is clear from the formula in Definition \ref{def canonical} and the definition \eqref{eqn steinberg weights dfn} of the $x_u$.
\end{proof}
%

%
%
%
%
%
%
%
%
\subsection{The Schwartz space of the basic affine space}
\label{subsection SI}
In \cite{BK}, Braverman-Kazhdan defined an embedding of $J$ into the Harish-Chandra Schwartz space of $G(F)$ and characterized its image. In particular, for $\bq=q$ a prime power one can study elements of $J$ as functions on $G(F)$ via the Plancherel formula and understanding the trace of $j\star T_w$ on tempered representations $\pi=i_P^G(\nu\otimes\omega)$, for $\omega$ a discrete series representation of the Levi subgroup of the parabolic subgroup $P$ of $G$ and $\nu$ an unramified character. The author used this strategy in \cite{Plancherel} to check that the map of \cite{BK} is an embedding. (Although we do not need this, the characterization of \cite{BK} is proven in \cite{BKK} and \cite{rigid}.)

To understand $J_0$, it suffices to consider principal series in the sense of the previous paragraph, by \cite[Thm. 1.8]{BK}. In \cite{BKBasicAffine}, Braverman-Kazhdan-Lusztig defined a space of functions $\mathcal{S}$ with commuting $G(F)$ and $T(F)$-actions, such that taking coinvariants of $\mathcal{S}$ at generic unramified characters of $T(F)$ gave an isomorphism with the corresponding principal series. As the traces of $\pi(j)$ for $\pi$ tempered and $j\in J$ extend to algebraic functions of the unramified character 
\cite[Lem. 2.17]{Plancherel}, we may compute in $\mathcal{S}^I$. This is a convenient model as 
\begin{equation}
\label{eqn basic affine isomorphism}
\mathcal{S}^I\simeq K(\Bb^\vee/T^\vee\times\Gm)|_{\bq=q}
\end{equation}
by \cite[Cor. 5.7]{BKBasicAffine}. 
\begin{rem}
\label{rem Calder}
Although we do not use the $W_f$-action by intertwining isomorphisms on $\mathcal{S}^I$-explicitly, we point out that its correct normalization is given in \cite{Calder}. 
In light of Example 6.3 of \textit{op. cit.} and the fact that $\mathcal{S}^K$ is spanned by $\Oo_{\Bb}$, the matching 
$W_f$-action the right hand side of \eqref{eqn basic affine isomorphism} is the ``dot" version of the action \cite[\S 5.2]{BKBasicAffine}.
\end{rem}

Diagram \eqref{eqn Xi diagram} induces the diagram
%
%
%
%
\begin{center}
\begin{tikzcd}
K_{G^\vee\times\Gm}(\St)=\HH\arrow[r]\arrow[d, hook, "\Phi"]&\End_{\Z[\tilde{W}]}\left(K\left(\Nn/T^\vee\times\Gm^\times\right)\right)\arrow[d, "\mathrm{Th}"]&
\\
K(\Bb^\vee\times\Bb^\vee/G^\vee\times\Gm)\simeq J_0\otimes\mathcal{A}\arrow[r]&\End_{\Z[\tilde{W}]}\left(K\left(\Bb^\vee/T^\vee\times\Gm\right)\right),
\end{tikzcd}
\end{center}
where $\mathrm{Th}$ is the Thom isomorphism \cite[Thm. 5.4.17]{CG}.

We identify $K(\Bb^\vee\times\Bb^\vee/G^\vee)\simeq K(\Bb^\vee/T^\vee)$ as $J_0$-modules via restriction to $\Bb^\vee\times\sett{\bB^\vee}$, so that the eigenspace of $\Oo(x_u)\boxtimes[\Oo(x_u)]^*$ is spanned by $\Oo(x_u)\bangles{[\Oo(x_u)]^*}$, where $\bangles{-}$ denotes the action of $K(\pt/T^\vee)$. This gives
\begin{lem}
\label{lem trace}
\begin{enumerate}
\item[(a)]
For any distinguished involution $d\in\cc_0$, and $\bq=q>1$, $t_d$ acts as a rank 1 idempotent on the unitary principal series of $G(F)$ and annihilates all other tempered representations.
\item[(b)]
If $t_w\in t_dJ_0t_d$ corresponds to $\lambda$ under \cite{Shi}, then 
$\trace{\pi}{t_w}=V(\lambda)$ for unramified principal series $\pi$ of $G(F)$, where we view $V(\lambda)$ as a function of the Satake parameter. 
\end{enumerate}
\end{lem}
\begin{proof}
Part (a) is \cite[Cor. 3.3]{Plancherel}. Part (b) then follows from (a).
\end{proof}

\subsubsection{Demazure-Lusztig operators and the Bernstein subalgebra}
%
The analogue of diagram (7.6.20) of \cite{CG} for our normalization gives that the action by Demazure-Lusztig operators is just the natural action of $\Phi(T_s)$ for $s\in S_{\mathrm{fin}}$. In our normalization they are given by \cite[Lem 4.7]{BasesI}, \cite[\S 1.2]{XiLowIII} and read
\[
T_{s_\alpha}\star\Oo(\lambda)=\Phi(T_{s_\alpha})\star\Oo(\lambda)=\frac{\Oo(s(\lambda))-\Oo(\lambda+\alpha)}{\Oo(\alpha)-1}+\bq\frac{\Oo(\lambda+\alpha)-\Oo(s(\lambda)+\alpha)}{\Oo(\alpha)-1}.
\]
\begin{lem}
\label{lem DL ops}
\begin{enumerate}
\item[(a)]
We have
\[
\Phi(T_s)\star\Oo=-\Oo,
\]
so that $\Phi({}^\dagger T_s)=\bq\Oo$.
\item[(b)]
If $x=s_\alpha$ is a simple reflection and $\pair{\lambda}{\alpha^\vee}\leq 0$, then 
\[
R\Gamma\left(\Oo(-\rho)\otimes\left(\Phi(T_s)\star\Oo(\lambda)\right)\right)
=
\bq R\Gamma\left(\Oo(\lambda-\rho)\right).
\]
\end{enumerate}
\end{lem}
\begin{proof}
The first statement follows from the definition in  \cite[Lem. 4.7]{BasesI}.

For the second statement, the non-$\bq$ term in $\Phi(T_s)\star\Oo(\lambda)$ is 
\[
\Oo(\lambda+\alpha)+\Oo(\lambda+2\alpha)+\cdots+\Oo(s(\lambda)-\alpha)).
\]
This term contributes no cohomology, as we have
\begin{equation}
\label{eqn s dot action DL}
s\cdot(\lambda+k\alpha-\rho)=s(\lambda)-k\alpha-\rho
\end{equation}
for $k\geq 1$. The $\bq$-term is
\[
-\Oo(\lambda+\alpha)-\Oo(\lambda+2\alpha)-\cdots-\Oo(s(\lambda)).
\]
We claim that 
\[
R\Gamma\left(\Oo(-\rho)\otimes\left(-\Oo(\lambda+\alpha)-\Oo(\lambda+2\alpha)-\cdots-\Oo(s(\lambda))\right)\right)=
R\Gamma(\Oo(-\rho+\lambda)).
\]
Indeed, using \eqref{eqn s dot action DL} again, we see that only $\Oo(s(\lambda)-\rho)$ contributes cohomology. Letting $x$ be such that $x(\lambda)$ is dominant, we get $\ell(xs)=\ell(x)-1$, and
\[
-R\Gamma\left(\Oo(s(\lambda)-\rho\right)=-(-1)^{\ell(xs)}V(x(\lambda)-\rho)=R\Gamma\left(\Oo(\lambda-\rho)\right).
\]
\end{proof}

Finally, we recall
\begin{lem}[\cite{CG}, Prop. 7.6.29]
\label{lem Lusztig map on Bernstein}
We have $\Phi(\theta_\lambda)=\Delta_*\Oo(\lambda)$ for $\Delta\colon\Bb^\vee\to\Bb^\vee\times\Bb^\vee$ the diagonal inclusion.
\end{lem}

\section{Formulas for the lowest two-sided cell}
\subsection{Proofs of Theorem \ref{thm K theory formulas} and Corollary \ref{cor symm and pos dc0}}
\label{subsection proof of thm for dc0}

\begin{proof}[Proof of Theorem \ref{thm K theory formulas}]
Let $\gamma$ be dominant and $t_w\in t_{d} J t_{d'}$ correspond to $\lambda$ dominant. Write $x_d$ for the Steinberg weight $x_u$ corresponding to $d$. Write $t_d=\Oo(x_d)\boxtimes[\Oo(x_d)]^*$
and suppose that
\[
t_{d'}=\Oo(x_{d'})\boxtimes[\Oo(x_{d'})]^*=\Oo(x_{d'})\boxtimes\Oo(y_{d'})[n(d')]
\]
as in \eqref{eqn based ring iso}.

The right hand side of Theorem \ref{thm K theory formulas} is a rational function of $\bq$, and the same is true of $a_{x,w}$ by \cite{Plancherel}. The logic of Section 3.3 of \textit{op. cit.} says that it therefore suffices to prove the theorem for $\bq=q$ for all $q>1$. Recall from Section \ref{subsection SI}, that after so specializing, we can view $t_w$ as a Harish-Chandra Schwartz function on $G(F)$. We can study functions $f$ via the Plancherel formula \cite{OpdamSpectral}, using that 
%
\[
(f\star T_{w^{-1}})(1)=\dInt{f(g)T_{w^{-1}}(g^{-1})}{g}=\vol(IwI)f(w).
\]
Recall from Section \ref{subsection SI} that $J_0$ annihilates all tempered representations except the unitary principal series, and that traces of elements of $J_0$ acting on tempered 
principal series can be computed in the model provided by \eqref{eqn basic affine isomorphism}.
%
%
 
Recall that to realize elements of $J$ as Schwartz functions on $G(F)$, we have to twist $\phi$ by the automorphism ${}^\dagger(-)$ defined in \cite[Def.1.6]{Plancherel}. 
%
We compute that
\begin{align}
t_{w}(\varpi^{\gamma})&=
q^{-\ell(\gamma)}\left(t_w\star T_{-\gamma}\right)(1)
\\
&=
q^{-\ell(\gamma)}\left(t_w\phi\left({}^\dagger T_{-\gamma}\right)\right)(1)
\\
&=
(-1)^{\ell(\omega(\gamma)_f)} q^{-\ell(\gamma)}\left(t_w\phi\left(q^{\ell(\gamma)} T_{\gamma}^{-1}\right)\right)(1)
\\
&=
(-1)^{\ell(\omega(\gamma)_f)}q^{-\frac{\ell(\gamma)}{2}}\left(t_w\phi\left(\theta_{-\gamma}\right)\right)(1)
\label{eqn traces tw at gamma x pre idempotent added}
\\
&=
(-1)^{\ell(\omega(\gamma)_f)}q^{-\frac{\ell(\gamma)}{2}}\left(t_w\phi\left(\theta_{-\gamma}\right)t_{dp}\right)(1)
\label{eqn traces tw at gamma x idempotent added}
\end{align}
where between lines \eqref{eqn traces tw at gamma x pre idempotent added} and \eqref{eqn traces tw at gamma x idempotent added} we used Lemma \ref{lem trace} (a).

We claim that \eqref{eqn traces tw at gamma x idempotent added} equals
\begin{equation}
\label{eqn tw integral trace}
(-1)^{\ell(\omega(\gamma)_f)}q^{-\frac{\ell(\gamma)}{2}}\dInt{\trace{\pi}{V(\lambda)\otimes \Oo(x_d)\boxtimes\Oo(y_{d'})[n(d')]\star\Delta_*\Oo\left(-\gamma\right)\star \Oo(x_d)\boxtimes[\Oo(x_d)]^*}}{\pi}.
\end{equation}

Indeed, the only possible eigenvectors of $t_w\star T_{-\gamma}$ lie in the image of $t_d$, and the upshot of Section \ref{subsection K-theoretic realizations} is that under the isomorphism of based rings \eqref{eqn based ring iso}, 
\[
t_{w}\star\phi\left(\theta_{-\gamma}\right)\star t_{d}
\]
is sent to
\begin{align}
&V(\lambda)\otimes\Oo(x_d)\boxtimes\Oo(y_{d'})[n(d')]
\star\Delta_*\Oo(-\gamma)\star\Oo(x_d)\boxtimes[\Oo(x_d)]^*
\\
&=
(-1)^{n(d')}V(\lambda)\otimes R\Gamma(\Oo(-\gamma+x_d+y_{d'}))\Oo(x_d)\boxtimes[\Oo(x_d)]^*
\nonumber
\\
&=
(-1)^{\ell(w_0)}(-1)^{n(d')}V(\lambda)\otimes V(-w_0(\gamma-x_d-y_{d'}-2\rho))\otimes\Oo(x_d)\boxtimes[\Oo(x_d)]^*,
\label{eqn tw trace rewrite}
\end{align}
where the last equality is by the Borel-Weil-Bott theorem, and we used the hypothesis that $\gamma$ is far from the walls.

Therefore, combining \eqref{eqn tw integral trace} and \eqref{eqn tw trace rewrite}, we have, by the Plancherel formula and Lemma \ref{lem trace} (b), 
\begin{align}
t_{w}(\varpi^{\gamma})
&=
q^{-\frac{\ell(\gamma)}{2}}(-1)^{\ell(\omega(\gamma)_f)+\ell(w_0)+n(d')}
\dIntOver{\mathrm{P.S.}}{V(\lambda)V\left(-w_0\left(\gamma-x_d-y_{d'}-2\rho\right)\right)}{\pi_{\mu_I}}
\nonumber
\\
&=
q^{-\frac{\ell(\gamma)}{2}}(-1)^{\ell(\omega(\gamma)_f)+\ell(w_0)+n(d')} \dIntOver{\mathrm{P.S.}}{V(\lambda)\overline{V\left(\gamma-x_d-y_{d'}-2\rho\right)}}{\pi_{\mu_I}}
\label{eqn Maconald inner product I Haar}
\\
&=
\frac{1}{\mu_I(K)}
q^{-\frac{\ell(\gamma)}{2}}
(-1)^{\ell(\omega(\gamma)_f)+\ell(w_0)+n(d')}
\dIntOver{\mathrm{P.S.}}{V(\lambda)\overline{V\left(\gamma-x_d-y_{d'}-2\rho\right)}}{\pi_{\mu_K}}
\label{eqn Maconald inner product K Haar}
\\
&=
\frac{1}{\mu_I(K)}
q^{-\frac{\ell(\gamma)}{2}}
(-1)^{\ell(\omega(\gamma)_f)+\ell(w_0)+n(d')}
P_W(q^{-1})
\pair{V(\lambda)}{V(\gamma-x_d-y_{d'}-2\rho)}_{q^{-1}}
\label{eqn Macdonald inner product}
\\
&=
\frac{1}{\mu_I(K)}
(-1)^{\ell(\omega(\gamma)_f)+\ell(w_0)+n(d')}
q^{-\frac{\ell(\gamma)}{2}}
\sum_{i}\dim\HomOver{G^\vee}{V(\lambda)}{V(\gamma-x_d-y_{d'}-2\rho)\otimes\Oo(\mathcal{N}^\vee)_i}q^{-i}.
\label{eqn tw(IgammaI)}\end{align}
Here the integral is taken over the tempered unramified principal series of $G(F)$. Between lines \eqref{eqn Maconald inner product I Haar} and \eqref{eqn Maconald inner product K Haar} we used that the Plancherel measure scales inversely to the Haar measure on $G(F)$, that $\mu_I(K)=P_W(q)=q^{-\ell(w_0)}P_W(q^{-1})$ for $P_W$ the Poincar\'{e} polynomial of $W$, and that and that
$\pair{-}{-}_{q^{-1}}$ is the specialization of the Macdonald inner product \cite[\S 3]{NelsonRam} normalized for $\mathrm{d}\pi_{\mu_K}$. Between lines \eqref{eqn Macdonald inner product} and \eqref{eqn tw(IgammaI)} we used Equation (3.13) and the equation following (3.10) in \textit{loc. cit.}.
This proves \eqref{tdJtd' dominant formula denom}.

To deduce \eqref{tdJtd' dominant formula no denom} from \eqref{tdJtd' dominant formula denom}, we used that the graded character of $\Oo(\g^\vee)^{G^\vee}$ is exactly 
\[
\frac{\prod_{i=1}^{r}(1-q^{-1})}{P_W(q^{-1})}
\]
by \cite[(3.11)]{NelsonRam}.

Finally, note that when $y_w$ exists, $x_w+y_w=r\cdot 0$ for some $r\in W_f$.

\end{proof}

\begin{proof}[Proof of Corollary \ref{cor symm and pos dc0}]
Part (a), follows from symmetry of the Macdonald inner product in \eqref{eqn Macdonald inner product} (corresponding via Parseval's theorem to the fact that the Kazhdan-Lusztig elements are real-valued functions on $G(F)$).
%

For (b), we need only note that by \cite[\S 4]{Plancherel}, if $j\in t_d J_0 t_d$, then $j$ is $\mathcal{P}\times I$-invariant, where $\mathcal{P}$ is the parahoric subgroup corresponding to $\bangles{\mathcal{L}(d)}\subset\Waff$.
\end{proof}

\subsection{Proof of Corollary \ref{cor tw0 Jtdc0}}
\begin{proof}[Proof of Corollary \ref{cor tw0 Jtdc0}]
By Lemma \ref{Lem dual classes}, $t_{d_{\cc_0}}$ corresponds to $\Oo(-\rho)\boxtimes\Oo(-\rho)[\ell(w_0)]$. Thus the first statement is just the specialization of Corollary \ref{cor symm and pos dc0} (a), as are parts (a) and (b) of the second statement.

Finally, we prove (c). Let $\gamma$ be dominant and let $x\in W_f$. Let $j\in t_{w_0}J t_{d_{\cc_0}}$. By \eqref{eqn based ring iso} there is a class $V_j\in K(\pt/G^\vee)$ such that we have
\begin{align}
j(\varpi^{x(\gamma)})
&=
q^{-\ell(\lambda)}j\star T_{-x(\gamma)}(1)
\nonumber
\\
&=
q^{-\ell(\lambda)}j\star \phi\left({}^\dagger T_{-x(\gamma)}\right)(1)
\nonumber
\\
&=
q^{-\ell(\lambda)}j\star \phi\left(T_{x(\gamma)}^{-1}\right)(1)
\label{eqn DemLusztig x pre iverse}
\\
&=
q^{-\ell(\lambda)}j\star \phi\left( T_{\gamma x^{-1}}^{-1}\right)(1)
\label{eqn DemLusztig x post inverse}
\\
&=
q^{-\ell(\lambda)}j\star \phi\left( T_xT_{\gamma}^{-1}\right)(1)
\label{eqn DemLusztig x post DemRearrange}
\\
&=
q^{-\frac{\ell(\gamma)}{2}} j\star \phi\left( T_x\theta_{-\gamma}\right)(1)
\label{eqn DemLusztig x pre DL action}
\\
&=
q^{-\frac{\ell(\gamma)}{2}}(-1)^{\ell(w_0)}\dIntOver{\mathrm{P.S.}}{\mathrm{trace}\left(\pi, V_j\otimes\left(\Oo\boxtimes\Oo(-\rho)\star T_x\star\Delta_*\Oo(-\gamma)\right)\right)}{\pi_{\mu_I}}
\\
&=
q^{-\frac{\ell(\gamma)}{2}}(-1)^{\ell(w_0)}\dIntOver{\mathrm{P.S.}}{\mathrm{trace}\left(\pi, V_j\otimes\left(\Oo\boxtimes\Oo(-\rho)\star T_x\star\Delta_*\Oo(-\gamma)\star T_x^{-1}\star\Oo\boxtimes[\Oo]^*\right)\right)}{\pi_{\mu_I}}
\\
&=
q^{-\frac{\ell(\gamma)}{2}}(-1)^{\ell(w_0)}(-1)^{\ell(x)}\dIntOver{\mathrm{P.S.}}{\mathrm{trace}\left(\pi, V_j\otimes\left(\Oo\boxtimes\Oo(-\rho)\star T_x\star\Delta_*\Oo(-\gamma)\star\Oo\boxtimes[\Oo]^*\right)\right)}{\pi_{\mu_I}}
\\
&=
q^{-\frac{\ell(\gamma)}{2}}(-1)^{\ell(w_0)}(-1)^{\ell(x)}\dIntOver{\mathrm{P.S.}}{\mathrm{trace}\left(\pi, V_j\otimes 
R\Gamma\left(\Oo(-\rho)\otimes\left(T_x\star\Oo(-\gamma)\right)\right)
\otimes\Oo\boxtimes[\Oo]^*
\right)}{\pi_{\mu_I}}
\\
&=
q^{-\frac{\ell(\gamma)}{2}}(-1)^{\ell(w_0)}(-1)^{\ell(x)}\dIntOver{\mathrm{P.S.}}{V_j\otimes 
R\Gamma\left(\Oo(-\rho)\otimes\left(T_x\star\Oo(-\gamma)\right)\right)}{\pi_{\mu_I}}
\label{eqn DemLusztig x pre DL action expanded}
\\
&=
q^{-\frac{\ell(\gamma)}{2}}(-1)^{\ell(w_0)}(-q)^{\ell(x)}\dIntOver{\mathrm{P.S.}}{V_j\otimes 
R\Gamma\left(\Oo(-\rho)\otimes\Oo(-\gamma)\right)}{\pi_{\mu_I}}
\label{eqn DemLusztig x post DL action}
\\
&=
(-q)^{\ell(x)}
j(\varpi^\gamma).
\nonumber
\end{align}

Between lines \eqref{eqn DemLusztig x pre iverse} and \eqref{eqn DemLusztig x post inverse} we used that 
$\ell(\gamma x^{-1})=\ell(\gamma)-\ell(x)$, so that 
\[
T_xT_{\gamma x^{-1}}=T_{x(\gamma)}
\]
whence 
\[
T_{x(\gamma)}^{-1}=T_{\gamma x^{-1}}^{-1}T_{x}^{-1}.
\]
Between lines \eqref{eqn DemLusztig x post inverse} and \eqref{eqn DemLusztig x post DemRearrange}, we used 
\[
T_{\gamma x^{-1}}T_x=T_{\gamma}, 
\]
which implies
\[
T_{\gamma x^{-1}}^{-1}=T_xT_{\gamma}^{-1}.
\]

Between lines \eqref{eqn DemLusztig x pre DL action} and \eqref{eqn DemLusztig x pre DL action expanded}, we expanded, used the same 
post-composition trick as in the proof of Theorem \ref{thm K theory formulas}, and applied Lemma \ref{lem DL ops} (a). 
Between lines \eqref{eqn DemLusztig x pre DL action expanded} and \eqref{eqn DemLusztig x post DL action}, we applied Lemma \ref{lem DL ops} (b). Indeed, writing 
$T_x=T_{s}T_{sx}$ for some $s$ and writing $\Phi(T_{sx})\star\Oo(-\gamma)$ as a linear combination of line bundles $\Oo(\lambda)$ with $\pair{\lambda}{\alpha^\vee}\leq 0$, the simplification follows by induction.
%
%
%
%
\end{proof}
\subsection{Proof of Corollary \ref{cor KL}: Spherical Kazhdan-Lusztig polynomials}
\label{subsection new basis}
\begin{proof}[Proof of Corollary \ref{cor KL}]
Let $t_w\in t_d Jt_{d'}$ where $d'$ satisfies the hypotheses of Theorem \ref{thm K theory formulas}. By taking linear combinations of such $t_w$, we may replace $V(\lambda)$ in \eqref{eqn Macdonald inner product} with (the Satake transform of) any other element of the spherical Hecke algebra. In particular, we may use the standard basis elements $q^{-\frac{\ell(\mu)}{2}}T_\mu$. Writing $P_\mu$ for the Hall-Littlewood polynomial 
\cite[(2.13)]{NelsonRam}, we obtain instead of \eqref{eqn Macdonald inner product} that there is a basis of $t_d Jt_{d'}$ consisting of functions $f_{\mu}=(-1)^{\ell(w_0)+n(d')}f'_\mu$, where 
\begin{align*}
f'_\mu(I\varpi^{\gamma} I)&=
\frac{1}{\mu_I(K)}
q^{-\frac{\ell(\gamma)}{2}}
(-1)^{\ell(\omega(\gamma)_f)+\ell(w_0)+n(d')}
P_W(q^{-1})
\pair{q^{-\frac{\ell(\mu)}{2}}P_\mu}{V(\gamma-x_d-y_{d'}-2\rho)}_{q^{-1}}
\\
&=
\frac{1}{\mu_I(K)}
q^{-\ell(\gamma)}
(-1)^{\ell(\omega(\gamma)_f)+\ell(w_0)+n(d')}
P_{\mu,\gamma-x_d-y_{d'}-2\rho}(q),
\end{align*}
where $P_{\mu,\gamma}$ is the spherical Kazhdan-Lusztig polynomial, and we used \cite[Thm. 3.17 (a), (d)]{NelsonRam}. 
\end{proof}

\section{Recurrence relations}
\label{section rec rel}
Let $X_*^+$ denote the dominant cocharacters. Casselman-Cely-Hales prove in \cite[Lem. 4.1.2]{CCH} that the function 
\[
P\colon X_*^+\times X_*^+\to \Z[q,q^{-1}]
\]
defined by $P(\mu,\gamma)= P_{\mu,\gamma}(q)$ is constructible with respect to the Presburger language, in the sense of model theory recalled in \cite[\S 3]{CCH}. By definition, Presburger-constructibility means that
\begin{equation}
\label{eqn Pres constructible definition}
P(\mu,\lambda)=P_{\mu,\lambda}(q)=\sum_{i=1}^{n}\frac{\prod_{j}\beta_{i,j}(\mu,\lambda)q^{\alpha_i(\mu,\lambda)}}{\prod_{k}(1-q^{a_{ik}})}
\end{equation}
for Presburger-definable functions $\alpha_i$ and $\beta_{i,j}$. 
In turn, definability of the $\alpha_i$ and $\beta_{i,j}$ means that there is a decomposition of $X_*^+\times X_*^+	$ into finitely-many subsets $C_1,\dots, C_r$ defined by linear inequalities and congruence conditions, such that for each $C=C_m$, all the functions $\beta_{i,j}|_C$ and $\alpha_i|_C$ are affine, \textit{i.e.} of the form
\[
\alpha_i(\mu,\lambda)=a^\lambda_i\cdot\lambda+a^\mu_i\cdot\mu+a_i^0.
\]
Note that the factor $\prod_{k}(1-q^{a_{ik}})$ is constant with respect to $\mu,\lambda$.

We want to understand the behaviour of any $f_\mu(N\lambda)$ for $N$ sufficiently large.
Writing \eqref{eqn Pres constructible definition} in coordinates $\nu=\sum_{i}\pair{\nu}{\varpi_i^\vee}\varpi_i$, $\nu_k=\pair{\nu}{\varpi^\vee_k}$, we obtain, for $(\mu,\nu)$ inside a fixed region $C$, 
\begin{equation}
\label{eqn presburger constructibility definition in coordiantes}
P_{\mu,\nu}(q)=\sum_{i=1}^{n}\prod_{j}\left(\sum_{k=1}^rb_{i,j,k}^1\nu_{k}+b_{i,j,k}^2\mu_k+b_{i,j}^0\right)q^{a_1^1\nu_1}q^{a_2^1\nu_1}\cdots q^{a_r^2\nu_r}q^{\pair{a_i^2}{\mu}+a_i^0}.
\end{equation}
In particular, inside each $C_i$, the function $\nu\mapsto P_{\mu,\nu}(q)$ satisfies finite linear recurrences with constant coefficients in all the directions $\varpi_1,\dots, \varpi_r$, as long as $\lambda+k\varpi_i$ remains in $C$. These recurrence relations depend only on the exponents $a^1_i$ and the degree in $\nu_i$ of the polynomials
\[
\prod_{j}\left(\sum_{k=1}^rb_{i,j,k}^1\nu_{k}+b_{i,j,k}^2\mu_k+b_{i,j}^0\right).
\]
The degrees of these polynomials, as polynomials in any $\nu_k$, depend only on $C$. Therefore the recurrence relations depend only on $C$ and the direction $\varpi_i$. In particular, they are independent of $\mu$ for large enough  $N$ and sufficiently large and generic $\lambda$, because then the chamber in $X_*^+$ containing $w\cdot(\lambda+N\varpi_i)-\mu$ is independent of $\mu$.

Hence the functions $\lambda\mapsto P_{\mu,\lambda}(q)$ satisfy recurrence relations along any affine ray $(\mu, \lambda+N\varpi_i)$ intersecting $C$. 
In particular, these functions are determined by finitely-many of their values in $C$.
\begin{prop}
\label{lem Presburger rec rel}
Fix $(\mu,\lambda)$ and let $B>0$ be such that there is $C_b$ such that $(\mu,N\lambda)\in C_b$ for all $N>B$. Then
there exists an integer $M=M(C_b)$ such that if the congruence class $N\equiv m\mod M$ of $N$ is fixed, we have
\begin{enumerate}
\item[(a)] 
\begin{equation}
\label{eqn Pres constructible rec rel}
P_{\mu,N\lambda}(q)=\sum_{i=1}^{n}A_{i,m}(N)(q^{a^\lambda_{i,b,m}})^N
\end{equation}
for polynomial functions $A_{i,b,m}$ and integers $a^\lambda_{i,b,m}$.
\item[(b)]
For $\lambda$ such that $(\mu,\nu+k\lambda)\in C$ for $k$  $N\equiv m\mod M$, the function
\[
N\mapsto f_\mu(N\lambda)=P_{\mu,N\lambda}(q)
\]
satisfies a finite linear recurrence relation in $N$ with constant coefficients. The recurrence relation depends only on the triple $(C_b, m, M(C_b))$.
\end{enumerate}
\end{prop}
\begin{proof}
Part (a) is just a rewriting of \eqref{eqn Pres constructible definition}. Part (b) follows as functions of the form \eqref{eqn Pres constructible rec rel} are precisely the functions eventually satisfying such recurrence relations.
\end{proof}
%
In type $A$ for $\mu=0$, the existence of these recurrence relations follows alternatively from the rationality of the generating function for stretched Kostka-Foulkes polynomials \cite[p. 91, \S7, Exercise 5]{Kirillov}.

\section{Other two-sided cells}
In this section $G=\GL_n$. Then two-sided cells $\cc$ for $\Waff$ are in bijection with Levi subgroups $M_\cc$ of $G$. When $M_\cc=\GL_m^{\times r}$ is homogeneous, we have a limited version of Theorem \ref{thm K theory formulas} for $\cc$. In this case, if $N=N(\cc)$ is the nilpotent corresponding to $\cc$ under Lusztig's bijection, then $Z_{G^\vee}(N)^{\mathrm{red}}=\GL_r$.
\begin{theorem}
\label{thm GLn homogeneous levi}
Suppose that $\cc$ corresponds to a partition with $r$ parts all equal to $m$.
Let $t_w\in t_d Jt_d$ with $d\in\cc$, and let $t_w$ correspond to $V(\lambda_r)$ under the isomorphism $t_dJt_d\simeq R(\GL_r)$ of based rings of \cite{XiTypeA}. Then 
\[
t_w(1)=\frac{d(\omega) P_{\GL_n/P_\cc}(q^{-1})}{P_{\Sn_r}(q^{-m})P_{\Sn_r}(q^{m})}
\sum_{i}\dim\HomOver{\GL_r}{V(\lambda_r)}{\Oo(\mathcal{N}^\vee_{\GL_r})_i}q^{-mi}=
\frac{d(\omega) P_{\GL_n/P_\cc}(q^{-1})}{P_{\Sn_r}(q^{-m})P_{\Sn_r}(q^{m})}
P^{\GL_r}_{0,\lambda_r}(q^m)
,
\]
where $P_{\GL_n/P_\cc}(q^{-1})$ is the Poincar\'{e} polynomial, $d(\omega)=d(\St_{\GL_n})^r$ is the formal degree of $\St_{\GL_m}^{\boxtimes r}$, and $P^{\GL_r}_{\mu,\lambda}$ is the spherical Kazhdan-Lusztig polynomial for $\GL_r$.
\end{theorem}
\begin{proof}
When $\cc$ is a two-sided cell for $\GL_n$ whose corresponding partition 
\[
(\underbrace{m, \dots, m}_r)
\]
has all parts equal, we have the equality of Harish-Chandra $\mu$-functions
\begin{equation}
\label{eqn mu coincidence}
\mu_{M,\GL_n}(z_1,z_2,\dots, z_r, q)=\mu_{T, \GL_r}(z_1,z_2,\dots, z_r, q^m)
\end{equation}
where $\mu_{T \GL_r}$ is the $\mu$-function in the summand of the Plancherel measure for $\GL_r(F)$ corresponding to the principal series. Recall also that $J_\cc$ annihilates all tempered representations except $\pi_\nu=i_{P_M}^G(\St\otimes\nu)$, and that Harish-Chandra's canonical measure $d\nu$ agrees for $M_P$ and $\GL_r$. (Indeed, in the notation of \cite[\S 2.6]{Plancherel}, we have $q_{ij}=1$ and $q^{ij}=q^{m}$ for all $i<j$.) 

Let $t_w=V(\lambda_r)t_d$ for $V(\lambda_r)\in R(\GL_r)$. Then the Plancherel theorem again gives, by \cite[Rem. 5.6]{AubertPlymen}, and the reasoning of Section \ref{subsection proof of thm for dc0},
\begin{align*}
t_w(1)&=d(\omega) P_{\GL_n/P_\cc}(q^{-1})\dIntOver{\X(M_\cc)\cdot\omega}{\trace{\pi_\nu}{t_w}\mu_{M, \GL_n}(\nu, q)}{\nu}
\\
&=\frac{d(\omega) P_{\GL_n/P_\cc}(q^{-1})}{P_{\Sn_r}(q^{-m})P_{\Sn_r}(q^{m})}\dIntOver{\mathrm{P.S.}}{V(\lambda_r)\mu_{T, \GL_r}(\nu, q^{m})}{\nu}
\\
&=
\frac{d(\omega) P_{\GL_n/P_\cc}(q^{-1})}{P_{\Sn_r}(q^{-m})P_{\Sn_r}(q^{m})}
\sum_{i}\dim\HomOver{\GL_r}{V(\lambda_r)}{\Oo(\mathcal{N}^\vee_{\GL_r})_i}q^{-mi},
\end{align*}
where $\X(M_\cc)\cdot\omega$ is the orbit of $\omega$ under the action of unramified characters of $M_\cc$, as recalled for instance in \cite[\S 2.2.3]{Plancherel}. Between the first and second lines, we rewrote the integral to be over the tempered principal series of $\GL_r(F')$, where $F'$ is any degree $m$ unramified extension of $F$, and rescaled the Haar measure of $\GL_r(F')$ to give $\GL_r(\Oo_{F'})$ unit volume, thereby rescaling the Plancherel measure. Between the second and third lines, we applied our trick from the proof of Theorem \ref{thm K theory formulas}. The last equality follows as in the proof of Corollary \ref{cor KL}.
\end{proof}
\begin{rem}
Whenever $M_\cc$ has blocks of different sizes, the coincidence \eqref{eqn mu coincidence} of $\mu$-functions fails.
\end{rem}

\begin{rem}
If $\lambda_r$ is spiral for $\GL_r$, in particular if $r=2$, the Kazhdan-Lusztig polynomial $P^{\GL_r}_{0,\lambda_r}$ is trivial.
\end{rem}


\bibliography{J0_positivity_paper_biblio.bib}

\end{document}